\documentclass{amsart} % please use amsart at 11pt
\usepackage{xcolor}
\usepackage{amssymb,latexsym}
\usepackage[OT2,T1]{fontenc}

\theoremstyle{plain}
\newtheorem{theorem}{Theorem}[section]

\newtheorem{corollary}[theorem]{Corollary}
\newtheorem{proposition}[theorem]{Proposition}

\theoremstyle{definition}
\newtheorem{definition}{Definition}[section]
\newtheorem{example}{Example}[section]
\newtheorem{question}{Question}[section]

\theoremstyle{remark}
\newtheorem{remark}{Remark}[section]

\DeclareMathOperator{\co}{co}%
\DeclareMathOperator{\diam}{diam}%
\newcommand{\N}{\mathbf N}

\newcommand{\R}{\mathbf R}

\newcommand{\f}{\varphi}
\newcommand{\e}{\varepsilon}

\newcommand{\skp}[2]{\left<#1,#2\right>}
\newcommand{\C}{\mathbf C}
\renewcommand{\a}{\alpha}

\newcommand{\NN}{\mathcal G}
\newcommand{\nn}{g}

\numberwithin{equation}{section} % to get equations numbered
\begin{document}
\title[Measures of noncompactness]{Measures of noncompactness on the standard Hilbert $C^*$-module} % please provide
                                % an abbreviated title

\author{Dragoljub J. Ke\v cki\' c}
\address{University of Belgrade\\ Faculty of Mathematics\\ Student\/ski trg 16-18\\ 11000 Beograd\\ Serbia}

\email{keckic@matf.bg.ac.rs}

\author{Zlatko Lazovi\'c}

\address{University of Belgrade\\ Faculty of Mathematics\\ Student\/ski trg 16-18\\ 11000 Beograd\\ Serbia}

\email{zlatkol@matf.bg.ac.rs}

\thanks{The authors was supported in part by the Ministry of education and science, Republic of Serbia, Grant \#174034.}

\begin{abstract}
We define a measure of noncompactness $\lambda$ on the standard Hilbert $C^*$-module $l^2(\mathcal A)$ over a unital $C^*$-algebra, such that $\lambda(E)=0$ if and only if $E$ is $\mathcal A$-precompact (i.e.\ it is $\e$-close to a finitely generated projective submodule for any $\e>0$) and derive its properties. Further, we consider the known, Kuratowski, Hausdorff and Istr\u{a}\c{t}escu measure of noncomapctnes on $l^2(\mathcal A)$ regarded as a locally convex space with respect to a suitable topology, and obtain their properties as well as some relationship between them and introduced measure of noncompactness $\lambda$.
\end{abstract}

%\dedicatory{This paper is dedicated to Professor X on his 125th birthday.}



\maketitle

\section{Introduction}

Measures of noncompactness (MNCs in further) have been studied almost for a century. Roughly speaking, a MNC
is a function which assigns a real number to any bounded set in a given metric space, and this real number
can be regarded as a characteristic of the extent to which $A$ is not totally bounded, (that is relatively compact when completeness is supposed). There are many different MNCs on metric spaces, among them the most cited are: Kuratowski, Hausdorf and Istr\u{a}\c{t}escu
MNC. Their definitions is given by:

\begin{definition} Let $(M,d)$ be a metric space, and let $A\subseteq M$ be a bounded set.

a) The Kuratowski measure of noncompactness of $A$, denoted by $\alpha(A)$, is the infimum of all $d>0$ such
that $A$ admits a partitioning into finitely many subsets whose diameters are less than $d$
(\cite{Kuratowski1}, see also \cite{Kuratowski2}).

b) The Hausdorff measure of noncompactness of $A$, denoted by $\chi(A)$, is the infimum of all $\e>0$ such
that $A$ has a finite $\e$-mash in $M$. (If we require that this $\e$-mash belongs to $A$, we refer to inner
Hausdorff measure of noncompactness, denoted by $\chi_i(A)$.)

c) The Istr\u{a}\c{t}escu measure of noncompactness, denoted by $I(A)$ , is the infimum of all $d>0$ such
that there is an $\e$-discrete sequence in $A$, that is, a sequence $x_n\in A$, such that $d(x_m,x_n)\ge\e$
for all $m\neq n$ (\cite{Istratescu}).
\end{definition}

Although different in general, these three functions share some common properties, e.g.\ the following:

\begin{proposition} Let $(M,d)$ be a metric space, and let $\mu$ be any of functions $\alpha$, $\chi$, $I$
defined above. Then $\mu$ has the following properties:

\begin{itemize}

\item[(a)] regularity: $\mu(A)=0$ iff $A$ is relatively compact;

\item[(b)] non-singularity: $\mu$ is equal to zero on any single-element set;

\item[(c)] monotonicity: $\mu(A)\le\mu(B)$, whenever $A\subseteq B$;

\item[(d)] subadditivity: $\mu(A\cup B)=\max\{\mu(A),\mu(B)\}$, for all $A$ and $B$;

\end{itemize}

If, in addition, $M$ is a normed space, i.e., $M$ is a linear space and the metric $d$ is defined via
$d(x,y)=||x-y||$ then $\mu$ has additional properties:

\begin{itemize}

\item[(e)] algebraic semi-additivity: $\mu(A+B)\leq\mu(A)+\mu(B)$, for all $A$, $B$;

\item[(f)] semi-homogeneity: $\mu(tA)=|t|\mu(A)$, for $t\in\C$ and all $A\subseteq M$;

\item[(g)] invariance under translations: $\mu(x+A)=\mu(A)$, for all $A\subseteq M$ and $x\in M$;

\item[(h)] Lipschitz continuity: $|\mu(A)-\mu(B)|\le L d_H(A,B)$, where $L$ is an absolute constant ($L_\alpha=2$, $L_\chi=1$, $L_I=2$) and $d(A,B)$ is the so called Hausdorff distance, that is $d_H(A,B)=\max\{\sup_{x\in A}d(x,B),\sup_{y\in B}d(y,A)\}$;

\item[(i)] invariance to the transition to the closure and to the convex hull: $\mu(A)=\mu(\overline
A)=\mu(\co A)$;

\item[(j)] The functions $\alpha$, $\chi$ and $I$ are equivalent to each other,
that is,
\begin{equation}\label{MNCinequalities}
\chi(A)\leq I(A)\leq\alpha(A)\leq2\chi(A)\quad\mbox{for all bounded }A.
\end{equation}
\end{itemize}

\end{proposition}

These properties are well known and their proofs can be found throughout literature, for instance in
\cite{MursaleenEtAl}. (Though, it should be mentioned that inner Hausdorff MNC does not satisfy the equality $\chi_i(\co A)=\chi_i(A)$, see \cite[page 9]{AkhmerovEtAl})

Some of the mentioned properties were singled out in order to establish the axiomatic definition of the
abstract notion of the MNC, and this was done in various different manners.

For more detailed exposition on MNCs on metric or normed spaces, the reader is referred to \cite{AkhmerovEtAl}, \cite{ToledanoEtAl} or \cite{Volodja}.

There is some extension of the MNC theory to the framework of uniform spaces which will be described in Section 3. Among all uniform spaces, we are specially interested in the standard Hilbert module $l^2(\mathcal A)$ over a unital $C^*$-algebra $\mathcal A$. It is defined by
$$l^2(\mathcal A)=\Big\{x=(\xi_1,\xi_2,\dots)\;|\;\xi_j\in\mathcal A,\sum_{j=1}^{+\infty}\xi^*_j\xi_j\mbox{
converges in the norm topology}\Big\},$$
and it is equipped with the $\mathcal A$-valued inner product
$$l^2(\mathcal A)\times l^2(\mathcal A)\ni(x,y)\mapsto\sum_{j=1}^{+\infty}\xi_j^*\eta_j\in\mathcal A,\qquad x=(\xi_1,\xi_2,\dots),\quad y=(\eta_1,\eta_2,\dots).$$

Since an arbitrary $\mathcal A$-linear bounded operator on $l^2(\mathcal A)$ does not need to have an
adjoint, the natural algebra of operators is $B^a(l^2(\mathcal A))$ - the algebra of all $\mathcal A$-linear
bounded operators on $l^2(\mathcal A)$ having an adjoint. It is known that $B^a(l^2(\mathcal A))$ is a
$C^*$-algebra, as well.

Among all operators in $B^a(l^2(\mathcal A))$, those that belong to the linear span of the operators of the
form $x\mapsto\Theta_{y,z}(x)=z\left<y,x\right>$ ($y$, $z\in l^2(\mathcal A)$) are called {\em finite rank
operators}. The norm closure of finite rank operators is known as the algebra of all "compact" operators. The quotation marks are usually written in order to emphasize the fact that "compact" operators does not maps
bounded sets into relatively compact sets, as it is the case in the framework of Hilbert (and also Banach)
spaces, though they share many properties of proper compact operators on a Hilbert space, \cite{Manuilov1},
\cite{Manuilov2}.

(For general literature concerning Hilbert modules over $C^*$ algebras, including the standard
Hilbert module, the reader is referred to \cite{Lance} or \cite{MT}.)

In our earlier work \cite{KeckicLazovic}, we construct a locally convex topology $\tau$ on $l^2(\mathcal A)$
such that $T\in B^a(l^2(\mathcal A))$ is "compact" implies that its image of the unit ball is totally
bounded with respect to $\tau$. The converse is obtained in the special case $\mathcal A=B(H)$.

%This topology is defined via the family of semi-norms $p_{\f,y}$
%
%\begin{equation}\label{polunorem intro}
%p_{\f,y}(x)=\sqrt{\sum_{j=1}^{\infty}|\f(\eta_j^*\xi_j)|^2},\quad x=(\xi_1,\xi_2,\dots)\in l^2(\mathcal A),
%\end{equation}
%
%where $\f\in\mathcal A_*$ is a normal state and $y=(\eta_1,\eta_2,...)$ is a sequence of elements in
%$\mathcal A$ such that
%
%\begin{equation}\label{uslov za polunormu intro}
%\sup_{j\geq1}\f(\eta_j^*\eta_j)=1.
%\end{equation}

The aim of this note is to introduce a MNC on the standard Hilbert module $l^2(\mathcal A)$ and to derive its properties. In Section 2, we introduce the MNC $\lambda$ such that for any bounded set $E\subseteq l^2(\mathcal A)$ the equality $\lambda(E)=0$ holds if and only if $E$ is precompact in the sense of \cite[Proposition 2.6]{MT}, which implies that the corresponding MNC of a given operator $T$ is equal to $0$ if and only if $T$ is "compact".

In Section 3 we list some known results on MNCs on uniform spaces and also derive some simple generalizations.

In Section 4, we discuss generalizations of Kuratowski, Hausdorff and Istr\u{a}\c{t}escu measures of noncomapctness on $l^2(\mathcal A)$ and obtain their relationship with MNC $\lambda$.

We use the following basic and also simple facts on Hilbert modules, that can be found
throughout the literature. To make proofs more easy we list and prove them.

(F1) Let $z_1\perp z_2$. Then $||z_1+z_2||\ge||z_1||$.

Indeed, we have
$$\skp{z_1+z_2}{z_1+z_2}=\skp{z_1}{z_1}+\skp{z_2}{z_2}\ge\skp{z_1}{z_1},$$
in the order defined by the positive cone in $A$. Therefore
$$||z_1+z_2||^2=||\skp{z_1+z_2}{z_1+z_2}||\ge||\skp{z_1}{z_1}||=||z_1||^2.$$

\smallskip

(F2) Let $M$ be a projective finitely generated submodule of $l^2(\mathcal A)$, and let $x\in l^2(\mathcal A)$ be arbitrary. Then
$$d(x,M)=||x-P_Mx||,$$
where $P_M$ is orthogonal projection onto $M$ with null-space $M^\perp$.

Indeed, by \cite[Theorem 1.4.5]{MT} $M$ is orthogonally complemented, $l^2(\mathcal A)=M\oplus M^\perp$, and there is
$P_M:l^2(\mathcal A)\to l^2(\mathcal A)$ such that $P_M^2=P_M^*=P_M$, the range of $P_M$ is $M$ and the kernel of
$P_M$ is $M^\perp$. Let $y\in M$ be arbitrary. Then
$$||x-y||=||(x-P_Mx)+(P_Mx-y)||\ge||x-P_Mx||,$$
by (F1), because $x-P_Mx\in M^\perp$ and $P_Mx-y\in M$.

\section{Measure of noncompactness $\lambda$}

Throughout this section, $\mathcal A$ will always denote a unital $C^*$-algebra, and its unit will be denoted
by $1$. Also, $l^2(\mathcal A)$ will denote the standard Hilbert $C^*$-module over $\mathcal A$ defined in
the introduction.

In \cite[Proposition 2.6]{MT} $\mathcal A$-precompact sets were defined as those bounded sets $E$ such that
for all $\e>0$ there is a free finitely generated module $M\cong\mathcal A^n$ such that
$$d(E,M):=\sup_{x\in E}d(x,M)=\sup_{x\in E}\inf_{y\in M}d(x,y)<\e.$$

We generalize this notion in the following way.

\begin{definition}\label{DefinicijaLambda}
Let $E\subseteq l^2(\mathcal A)$ be a bounded set. The measure of noncompactness of $E$,
denoted by $\lambda(E)$ is the greatest lower bound of all $\eta>0$ for which there is a  free finitely
generated module $M\le l^2(\mathcal A)$ such that
$$d(E,M):=\sup_{x\in E}\inf_{y\in M}d(x,y)<\eta.$$
\end{definition}

\begin{proposition}\label{Calc}
The measure of noncompactness $\lambda(E)$ can be computed as:
\begin{enumerate}
\item\label{Calc1} $\lambda(E)=\inf_{M\in\mathcal F}\sup_{x\in E}d(x,M)$, where $\mathcal F$ is the set of all free finitely generated modules;
\item\label{Calc2} $\lambda(E)=\lim_{n\to+\infty}\sup_{x\in E}||x-P_nx||=\inf_{n\ge1}\sup_{x\in E}||x-P_nx||$, where
$P_n:l^2(\mathcal A)\to l^2(\mathcal A)$ is given by $P_n(x_1,x_2,\dots)=(x_1,x_2,\dots,x_n,0,0,\dots)$.
\end{enumerate}
\end{proposition}

\begin{proof} The equation (\ref{Calc1}) is obvious. The sequence $I-P_n$ is decreasing. Hence $\lim_{n\to+\infty}\sup_{x\in E}||x-P_nx||=\inf_{n\ge1}\sup_{x\in E}||x-P_nx||$. Since
$P_n l^2(\mathcal A)$ is a free finitely generated module, we have immediately
$$\lambda(E)\le\inf_{n\ge1}\sup_{x\in E}||x-P_nx||.$$
To get the opposite inequality, let $\e>0$. Then, there is a free finitely generated module $M$ such that
$d(x,M)<\lambda(E)+\e$. Denote the projection on $M$ by $Q$. By \cite[Proposition 2.2.1.]{MT},
$||Q-P_nQ||\to0$ as $n\to+\infty$. Since $P_nQx\in P_nl^2(\mathcal A)$, we have by (F2)
\begin{align*}
||x-P_nx||=&d(x,P_nl^2(\mathcal A))\le||x-P_nQx||\le\\
    \le&||x-Qx||+||Qx-P_nQx||\le\lambda(E)+\e+K||Q-P_nQ||,
\end{align*}
where $||x||\le K$ for all $x\in E$ ($E$ is bounded). Thus
$$\sup_{x\in E}||x-P_nx||\le\lambda(E)+\e+K||Q-P_nQ||\to\lambda(E)+\e,\qquad\mbox{as }n\to+\infty,$$
which finishes the proof.
\end{proof}

\begin{proposition}\label{AlgPro} The measure of noncompactness $\lambda$ has the following properties
\begin{enumerate}
\item\label{AlgPro1} if $E\subseteq F$ then $\lambda(E)\le\lambda(F)$;

\item\label{AlgPro2} $\lambda(E\cap F)\le\min\{\lambda(E),\lambda(F)\}$;

\item\label{AlgPro3} $\lambda(E\cup F)\le\max\{\lambda(E),\lambda(F)\}$;

\item\label{AlgPro4} $\lambda(E+F)\le\lambda(E)+\lambda(F)$.

\item\label{AlgPro5} $\lambda(Ea)\le\lambda(E)||a||$, where $a\in A$. If, in addition, $a$ is invertible,
then also $||a^{-1}||^{-1}\lambda(E)\le\lambda(Ea)$. In particular, when $a$ is unitary then
$\lambda(Ea)=\lambda(E)$.

\item\label{AlgPro6} $\lambda(\co E)=\lambda(E)$, where $\co E=\{\sum_{i=1}^nt_ix_i\:|\:0\le
t_i\in\R,\sum_{i=1}^nt_i=1,x_i\in E\}$ is the convex hull of $E$.
\end{enumerate}
\end{proposition}

\begin{proof} (\ref{AlgPro1}) It is obvious.

(\ref{AlgPro2}) This follows from (\ref{AlgPro1}).

(\ref{AlgPro3}) Let $d=\max\{\lambda(E),\lambda(F)\}$. Then for all $x\in E$, as well as for all $x\in F$, we
have $||x-P_nx||\le d+\e$, for all $\e>0$ and $n$ large enough. Hence the result.

(\ref{AlgPro4}) Let $z\in E+F$. Then $z=x+y$ for some $x\in E$, $y\in F$. We have
$$||z-P_nz||\le||x-P_nx||+||y-P_ny||\le\lambda(E)+\e+\lambda(F)+\e,$$
for $n$ large enough.

(\ref{AlgPro5}) Any $z\in Ea$ is of the form $z=ya$ for some $y\in E$. Therefore
$||z-P_nz||\le||y-P_ny||\,||a||$ and the inequality follows by taking a limit. If $a$ has the inverse
$a^{-1}$, then $E=(Ea)a^{-1}$ and by previous $\lambda(E)\le\lambda(Ea)||a^{-1}||$.

(\ref{AlgPro6}) Let $x\in\co E$. Then $x=\sum_{i=1}^nt_ix_i$ for some $x_i\in E$, and positive $t_i$ such
that $\sum_{i=1}^nt_i=1$. We have
$$||x-P_nx||=\Big\|\sum_{i=1}^nt_i(x_i-P_nx_i)\Big\|\le\sum_{i=1}^nt_i||x_i-P_nx_i||\le\sup_{x\in E}||x-P_nx||.$$
Thus, $\sup_{x\in\co E}||x-P_nx||\le\sup_{x\in E}||x-P_nx||$. The opposite inequality is obvious. The
required follows from Proposition \ref{Calc}-(\ref{Calc2}).
\end{proof}

\begin{proposition}\label{UnitBall} Let $B$ denote the unit ball in $l^2(\mathcal A)$. Then $\lambda(B)=1$.
\end{proposition}

\begin{proof} Any submodule contains the origin. Hence $\lambda(B)\le 1$. Let $0<\delta<1$, and let $M$ be
some free finitely generated submodule of $l^2(\mathcal A)$. Then, there is nontrivial $y\in M^\perp$ and
$\delta ||y||^{-1}y\in B\cap M^\perp$. We have
$$d(B,M)\ge d(\delta||y||^{-1}y,M).$$
However, $\delta||y||^{-1}y\perp M$ which implies that for all $x\in M$ we have
$$||\delta||y||^{-1}y-x||^2=\delta^2+||x||^2\ge\delta^2.$$
Hence $d(B,M)\ge\delta$. Thus $\lambda(B)\ge\delta$.
\end{proof}

\begin{corollary}\label{Cor} If $E\subseteq F+\delta B$ then $\lambda(E)\le\lambda(F)+\delta$.
\end{corollary}

\begin{proof} Immediately follows from Proposition \ref{AlgPro}-(\ref{AlgPro4}) and Proposition
\ref{UnitBall}.
\end{proof}

\begin{proposition}\label{ConPro} The measure of noncompactness $\lambda$ has the following continuity properties:
\begin{enumerate}
\item\label{CnPro1} $|\lambda(E)-\lambda(F)|\le d_H(E,F)=\max\{d(E,F),d(F,E)\}$ ($d_H$ stands for the so called
Hausdorf distance.)

\item\label{CnPro2} $\lambda(E)=\lambda(\overline E)$ ($\overline E$ stands for the norm closure of $E$);

\item\label{CnPro3} $\lambda(E)=0$ iff $E$ is $\mathcal A$-precompact;

\item\label{CnPro4} $\lambda(E)\le\sup_{x\in E}||x||$.
\end{enumerate}
\end{proposition}

\begin{proof} (\ref{CnPro1}) Let $d=d_H(E,F)$. Then $E\subseteq F+dB$ and by Corollary \ref{Cor}
$$\lambda(E)\le\lambda(F)+d,\qquad\mbox{i.e.}\qquad\lambda(E)-\lambda(F)\le d.$$
Similarly, $F\subseteq E+dB$ implying $\lambda(F)-\lambda(E)\le d$.

(\ref{CnPro2}) As it is easy to see $d_H(E,\overline E)=0$ we can apply the previous item.

(\ref{CnPro3}) Follows directly from the definition.

(\ref{CnPro4}) Follows from $E\subseteq(\sup_{x\in E}||x||)\cdot B$.
\end{proof}

\begin{example} In Proposition \ref{AlgPro}-(\ref{AlgPro5}) the strict inequalities might hold.

Indeed, let the algebra $\mathcal A$ contain a nontrivial projection, say $p$, and let $\mathcal A(1-p)$ is
isomorphic to $\mathcal A$. (For instance $\mathcal A=L^\infty(0,1)$ and $p=\chi_{[0,1/2]}$.) Let
$$E=\Big\{(a_1p+b_1(1-p),a_2p+b_2(1-p),\dots)\:|\:\sum_{n=1}^{\infty}a_n^*a_n\leq1,\sum_{n=1}^{\infty}b_n^*b_n\leq4\Big\}$$
and let $a=p$. Then $||p||=1$ and $\lambda(E)\ge2$, since $E$ contains a copy of a ball of radius $2$ in
$l^2(\mathcal A)$ (when $a_j=0$). On the other hand $\lambda (Ep)\le1$. Indeed, $Ep$ is contained in the unit
ball of $l^2(\mathcal A)$.
\end{example}

Finally, we want to define MNC of an operator $T\in B^a(l^2(\mathcal A))$. As it is expected, it will be the
MNC of its image of the unit ball.

\begin{definition}\label{DefinicijaLambda0}
Let $T\in B^a(l^2(\mathcal A))$ be an adjointable operator. We set
$$\lambda_0(T)=\lambda(T(B_1)),$$
where $B_1$ is the unit ball in $l^2(\mathcal A)$.
\end{definition}

\begin{proposition} The function $\lambda_0$ has the following properties:

\begin{itemize}

\item[(a)] $\lambda_0$ is subadditive, i.e.
$$\lambda_0(T_1+T_2)\leq\lambda_0(T_1)+\lambda_0(T_2);$$

\item[(b)] $\lambda_0$ is positively homogeneous, i.e.
$$\lambda_0(cT)=c\lambda_0(T),$$
for all $c>0$ and all $T\in B^a(l^2(\mathcal A))$;

\item[(c)] $\lambda_0(T)\leq||T||$, for all $T\in B^a(l^2(\mathcal A))$;
\end{itemize}
\end{proposition}

\begin{proof} Direct verification.
\end{proof}

We will be able to say more on the MNC $\lambda_0$ in section 4.

\section{Measure of noncompactness on uniform spaces - known results}

Let us recall some basic definitions and facts concerning uniform spaces. For
more details see \cite{Sadovskii} or \cite{Bourbaki}.

Uniform spaces are those topological spaces in which one can deal with notions such as Cauchy sequence,
Cauchy net or uniform continuity. Although it is usual to define them as spaces endowed with a family of sets
in $X\times X$ given as some kind of neighborhoods of the diagonal, so called entourages, for our purpose it
is more convenient to give an equivalent definition, via a family of semi-metrics.

\begin{definition}  A nonempty set endowed with a family of semi-metrics, functions
$d_{\alpha}\colon X\times X\rightarrow[0,+\infty)$ satisfying (i) $d_{\alpha}(x, y)\geq0;$ (ii) $d_{\alpha}(x,y)=d_{\alpha}(y,x);$ (iii)
$d_{\alpha}(x, z)\leq d_{\alpha}(x, y)+d_{\alpha}(y, z)$ is called a uniform space.
\end{definition}

\begin{remark} There is some ambiguity in literature; sometimes, functions $d_\a$ from the previous
definition are called "pseudo-metrics", whereas the term "semi-metric" is reserved for a different notion.
\end{remark}

All $d_{\alpha}$ are metrics except they do not distinguish points, i.e.\ there might be $d_{\alpha}(x,y)=0$
for some $x=y$. However it is provided that for all $x=y$ there is an $\alpha$ such that $d_{\alpha}(x, y) >
0$. The family of sets $B_{d_{\alpha}}(x;\varepsilon)=\{y\in X| d_{\alpha}(x,y) < \varepsilon\}$ makes a
basis for some topology. It is well known that a topological space X is a uniform space if and only if it is
completely regular.

%Let X be a uniform space. For a net $x_i\in X$ we say that it is a Cauchy net if it is a Cauchy net with
%respect to all $d_{\alpha}$, i.e. if for all $\alpha$ and for all $\varepsilon>0$ there is $i_0$ such that
%for all $i,j > i_0$ we have $d_{\alpha}(x_i,x_j)<\varepsilon$. The notion of a complete uniform space is
%defined in an obvious way. A set $A\subset X$ is called totally bounded, if for all $\varepsilon>0$ and all
%$\alpha$ there is a finite set $c_1, c_2, . . . , c_m\in X$ such that $B_{\alpha}(c_j;\varepsilon)=\{y\in %X|
%d_{\alpha}(c_j,y)<\varepsilon \}$ cover $A$. It is well known that any relatively compact set is totally
%bounded, and that the converse is true provided that $X$ is complete. If X is not complete then there are
%totally bounded sets that are not relatively compact, for instance, $\mathbb{Q}\cap [0, 1]$ as a subset of
%$Q$ (see also \cite[Remark 4.2.2]{Bourbaki}).

Any locally convex topological vector space is a uniform space. Indeed, there is a family of semi-norms
generating its topology. This family can be obtained by Minkowski functionals of basic neighborhoods of zero. And an arbitrary semi-norm define a semi-metric in a natural way. Conversely, any family of semi-norms that
distinguishes points leads to a locally convex Hausdorff topological vector space.% Hence a family of
%semi-norms allows to deal with notions: totally bounded set, complete space, Cauchy net, etc.

We point out two generalizations of the notion of MNC to the framework of uniform spaces, i.e, those
topological spaces that arise from the family of semi-metrics.

Sadovskii \cite{Sadovskii} considered a uniform space $X$ and a family $P$ of semi-metrics that are uniformly continuous on $X\times X$. Starting with Kuratowski and Hausdorff MNC on a semi-metric space (which is defined exactly as on a metric space), Sadovskii defined corresponding MNCs $\alpha$, $\chi:\Upsilon\to\NN$, where $\Upsilon$ denotes the family of all subsets that are bounded with respect to any semi-metric $p\in P$ on the given uniform space, and $\NN$ denotes the set of functions $\nn:P\rightarrow[0,+\infty)$ with uniformity generated by pointwise convergence and the natural partial ordering: $\nn_1\leq \nn_2\Leftrightarrow (\forall p\in P)\;\nn_1(p)\leq \nn_2(p)$. Their definitions are
$$[\alpha(E)](p)=\inf\Big\{d > 0\;:\;E=\bigcup_{j=1}^mE_j,\mbox{ for some }E_j,\diam(E_j)<d\Big\},$$
$$[\chi(E)](p)=\inf\Big\{\varepsilon>0\;:\;E\subseteq\bigcup_{j=1}^mB_p(x_j;\varepsilon)\mbox{ for some }x_j\in X\Big\}.$$
For such defined $\a$ and $\chi$, in \cite[\S1.2.3. and \S1.2.5.]{Sadovskii}, the following is proved,
provided that the family $P$ generates the topology on $X$.

\begin{theorem}\label{osobine mere nekompaktnosti}
The Kuratowski and Hausdorff measures of noncompactness ($\mu=\alpha$ or $\mu=\chi$) have the following
properties:

\begin{itemize}
\item[(a)] $\mu$ is non-singular, that is, they are zero on any
single-element set;

\item[(b)] $\mu$ is continuous, that is, for all $E\in\Upsilon$, $p\in P$ and $\varepsilon>0$ there is an
entourage $V$ in $X$ such that for all $E_1$ that is $V$-close to $E$ there holds
$|\mu(E_1)(p)-\mu(E)(p)|<\varepsilon$;

\item[(c)]$\mu$ is semi-additive, that is, for all $E_1$, $E_2$ we have
$$\mu(E_1\cup E_2)=\max\{\mu(E_1),\mu(E_2)\};$$

\item[(d)] The function is algebraically semi-additive, that is,
$$\mu(E_1+E_2)(p)\leq\mu(E_1)(p)+\mu(E_2)(p)\quad\mbox{for all }E_1,E_2\in\Upsilon;$$

\item[(e)] $\mu$ is invariant under shifts, that is,
$$\mu(x+E)=\mu(E)\quad\mbox{for all }E\in\Upsilon, x\in X;$$

\item[(f)] $\mu$ is invariant under the transition to its closure and to the convex hull of
the set, that is,
$$\mu(\co E)=\mu(\overline E)=\mu(E)\quad\mbox{for all }E\in\Upsilon;$$

\item[(g)] $\mu$ is uniformly continuous, that is, for all $p\in P$ and $\varepsilon>0$ there is an entourage
$V$ in $X$ such that for all $V$-close $E_1$ and $E_2$ there holds $|\mu(E_1)(p)-\mu(E_2)(p)|<\varepsilon$;

\item[(h)] The functions $\alpha$, $\chi$ and $I$ are equivalent to each other,
that is,
$$\chi(\Omega)\leq I(\Omega)\leq\alpha(\Omega)\leq2\chi(\Omega)\quad\mbox{for all}\,\,\Omega\in\Upsilon.$$
\end{itemize}
\end{theorem}

Although it was not done in \cite{Sadovskii}, one can also define the Istr\u{a}\c{t}escu MNC in a similar
way, i.e.
$$[I(E)](p)=\inf\{\varepsilon>0\;:\;E\ni x_n,\;p(x_n,x_m)\ge\varepsilon\mbox{ for all }m\neq n\}.$$
By (\ref{MNCinequalities}) one can easily derive
\begin{equation}\label{MNCinequalities1}
[\chi(E)](p)\leq[I(E)](p)\leq[\alpha(E)](p)\leq2[\chi(E)](p)\quad\mbox{for all bounded }E.
\end{equation}
Also, Theorem \ref{osobine mere nekompaktnosti} hold for $\mu=I$. Part (a) is obvious, parts (d) and (e)
follow from \cite[Proposition 1]{IvanFilomat}, part (c) follows from \cite{Istratescu}, parts (b) and (g)
from (\ref{MNCinequalities1}) and the corresponding properties of $\chi$, and finally (f) can be derived from
\cite[Theorem 1.3.4]{AkhmerovEtAl}.

\smallskip

Arandjelovi\'c \cite{IvanMoravica} dealt with an arbitrary uniform space, and gave the following:

\begin{definition}\label{IvanAxiome}
Let $X$ be a uniform space, metric space or semi-metric space. Any function $\Phi$ defined on the partitive
set of $X$, which satisfies the following:

\begin{itemize}

\item[(1)] $\Phi(E)=+\infty$ if and only if $E$ is unbounded;

\item[(2)] $\Phi(E)=\Phi(\overline E)$;

\item[(3)] from $\Phi(E)=0$ follows that $E$ is totally bounded set;

\item[(4)] from $E\subseteq F$ it follows $\Phi(E)\le\Phi(F)$;

\item[(5)] if $X$ is complete, and if $\{E_n\}_{n\in\N}$ is a sequence of closed subsets of $X$ such that
$E_{n+1}\subseteq E_n$ for each $n\in\N$ and $\lim_{n\to\infty}\Phi(E_n)=0$, then $K=\bigcap_{n\in\N}E_n$ is
a nonempty compact set.

\end{itemize}
is called a measure of noncompactness on $X$.
\end{definition}

\begin{remark}\label{MarininUslov}
Note that the only nontrivial requirement in (5) is that $K$ is nonempty. Moreover, condition (5) can be
replaced by a weaker one - $\Phi(A\cup\{x\})=\Phi(A)$. It was shown in \cite{Marina}, see also \cite{}.
\end{remark}

\begin{theorem}\label{Arandjeloviceva mera}\cite[Theorem 3]{IvanMoravica}
Let $X$ be a uniform space and let $\{d_i|i\in I\}$ be a family of semi-metrics which defines topology on
$X$. Denote by $\mu_i$ arbitrary MNC on the semi-metric space $(X,d_i)$ for each $i\in I$. Then the function
$\mu^*\colon X\rightarrow[0,+\infty]$ defined by
$$\mu^*(E)=\sup_{i\in I} \mu_i(E)$$
for each $E\in X,$ is a measure of noncompactness on $X$.
\end{theorem}

Uniform spaces make a proper subclass of all topological spaces, but still wide enough. For instance all
topological vector spaces are uniform spaces. Hence, we can apply results of Arandjelovi\'c to topological vector spaces.

\section{Measures of "noncompactness" over standard Hilbert $W^*$-modul $l^2(B(H))$}

In this section we shall discuss standard Hilbert modules over a $W^*$-algebra $\mathcal A$, a narrower class
then that considered in section 2. As it is well known, $\mathcal A$ always has a unit.

In our earlier work \cite{KeckicLazovic}, we construct a locally convex topology $\tau$ on $l^2(\mathcal A)$
such that $T\in B^a(l^2(\mathcal A))$ is "compact" implies its image of the unit ball is totally bounded with
respect to $\tau$. This topology is defined via the family of semi-norms $p_{\f,y}$
\begin{equation}\label{polunorme}
p_{\f,y}(x)=\sqrt{\sum_{j=1}^{\infty}|\f(\eta_j^*\xi_j)|^2},\quad x=(\xi_1,\xi_2,\dots)\in l^2(\mathcal A),
\end{equation}
where $\f\in\mathcal A_*$ is a normal state and $y=(\eta_1,\eta_2,...)$ is a sequence of elements in
$\mathcal A$ such that
\begin{equation}\label{uslov za polunormu}
\sup_{j\geq1}\f(\eta_j^*\eta_j)=1.
\end{equation}

Also, in special case, where $\mathcal A=B(H)$ is the full algebra of all bounded operators on a Hilbert
space $H$, the converse is also proved, i.e.\ that any $T\in B^a(l^2(\mathcal A))$ whose image of the unit
ball is totally bounded with respect to $\tau$ must be "compact".

Construction described in the previous section endows the space $l^2(\mathcal A)$ by the corresponding Kuratowski,
Hausdorff and Istr\u{a}\c{t}escu measure of noncompactness, $\alpha,\chi,I:\Upsilon\to\NN$,
$$[\alpha(E)](p_{\f,y})=\inf\Big\{\varepsilon>0:E=\bigcup_{i=1}^{n} S_i,\;p_{\f,y}(x'-x'')<\varepsilon,\,\,\forall x',x''\in S_i\Big\},$$
$$[\chi(E)](p_{\f,y})=\inf\Big\{\varepsilon>0: E\subset\bigcup_{i=1}^{n} B_{p_{\f,y}}(x_i,\varepsilon), x_i\in l^2(\mathcal A)\Big\},$$
where $B_{p_{\f,y}}(x_i,\varepsilon)=\{y|p_{\f,y}(y-x_i)<\varepsilon\}$, and
$$[I(E)](p_{\f,y})=\sup\{\varepsilon>0: \hbox{there is}\;S\subset E\;\hbox{such that}\;p_{\f,y}(x'-x'')\geq \varepsilon,\;\forall x',x''\in S\}.$$

The function $\a$, $\chi$ and $I$ can be regarded as functions depending on two variables, on the bounded set
$\Omega$ and on the semi-norm $p_{\f,y}$. If we want to obtain a MNC that not depends on a particular
semi-norm, we can use the functions $\chi^*$, $\alpha^*$, $I^*:\Upsilon\rightarrow[0,+\infty)$ defined by
\begin{equation}\label{MNCdef}
\begin{gathered}
\chi^*(E)=\sup_{p_{\f,y}\in P} [\chi(E)](p_{\f,y}),\\
\alpha^*(E)=\sup_{p_{\f,y}\in P} [\alpha(E)](p_{\f,y}),\\
I^*(E)=\sup_{p_{\f,y}\in P} [I(E)](p_{\f,y})
\end{gathered}
\end{equation}
for each $E\in \Upsilon$, where $P$ is the set of all semi-norms of the form (\ref{polunorme}). Since
$\alpha$, $\chi$ and $I$ annihilates all singletons (Theorem \ref{osobine mere nekompaktnosti}-(a)), they
satisfy condition in Remark \ref{MarininUslov}, and hence they are measures of noncompactness in the sense of
Definition \ref{IvanAxiome}. By Theorem \ref{Arandjeloviceva mera} $\alpha^*$, $\chi^*$ and $I^*$ are
measures of noncompactness on $(l^2(\mathcal A),\tau)$ in the sense of Definition \ref{IvanAxiome}, as well.
Also, properties (c), (d), (e) and (f) in Theorem \ref{osobine mere nekompaktnosti} are easily transferred to
$\alpha^*$, $\chi^*$ and $I^*$, by taking a supremum. Finally, by (\ref{MNCinequalities1}), we have
\begin{equation}\label{chainMNC}
\chi^*(E)\leq I^*(E)\leq\alpha^*(E)\leq2\chi^*(E).
\end{equation}

\begin{remark} Note that $l^2(\mathcal A)$ is rarely complete, due to \cite[Proposition 3.3]{KeckicLazovic}. Therefore, the condition (5) in Definition \ref{IvanAxiome} is vague, unless $l^2(\mathcal A)'\cong l^2(\mathcal A)$ which is equivalent to the condition that $\mathcal A$ is finite dimensional.
\end{remark}

We want to place the MNC $\lambda$, discussed in section 2, somewhere in the preceding chain of inequalities.

\begin{proposition}\label{chi_leq_lambda}
For any bounded set $E\subseteq l^2(\mathcal A)$, we have
$$\chi^*(E)\leq \lambda(E).$$
\end{proposition}

\begin{proof} Let $E$ be a bounded set, and let $P_n$ denote the projection to the first $n$ coordinates
in $l^2(\mathcal A)$, i.e.\ $P_n(\xi_1,\xi_2,\dots)=(\xi_1,\xi_2,\dots,\xi_n,0,\dots)$. Since $E\subseteq
P_nE+(I-P_n)E$, and since $\chi^*$ is subadditive, we have
$$\chi^*(E)\leq\chi^*(P_nE)+\chi^*((I-P_n)E).$$
However, by (\cite[Proposition 3.4.]{KeckicLazovic}), the set $P_nE$ is totally bounded, and we have
$\chi^*(P_nE)=0$. Hence
$$\chi^*(E)\leq\chi^*((I-P_n)E)\leq\sup_{x\in E}||(I-P_n)x||,$$
for all $n\in\mathbb{N}$. Therefore, by Proposition \ref{Calc}-(\ref{Calc2}), $\chi^*(E)\leq\lambda(E).$
\end{proof}

The preceding Proposition establishes a lower bound of $\lambda$. Before we obtained an upper bound for $\lambda$, in a special case, we introduce balanced sets.

\begin{definition} Let $E\subset l^2(\mathcal A)$ be a bounded set.

\begin{itemize}
\item[(a)] We say that $E$ is \emph{$\mathcal A$-balanced} if $x\cdot u\in E$ whenever $x\in E$ and $u\in\mathcal A$ is unitary. (This definition is motivated by the notion of balanced sets on topological vector spaces over the field $\mathbf C$, where $u$ unitary is reduced to $|u|=1$.)

\item[(b)] By \emph{$\mathcal A$-balanced hull} of $E$ we assume the minimal balanced set containing $E$, that is $\bigcup Eu$, where the union is taken over all unitaries $u\in\mathcal A$.

\end{itemize}
\end{definition}

In Proposition \ref{AlgPro}-(\ref{AlgPro5}) we proved $\lambda (Eu)=\lambda(E)$. We give two extensions of this statement, the first of them concerning the balanced hull.

\begin{proposition} Let $E\subseteq l^2(\mathcal A)$ be a bounded set and let $F$ be its $\mathcal A$-balanced hull. Then $\lambda(F)=\lambda(E)$.
\end{proposition}

\begin{proof} For all $x\in E$ and all unitaries $u$, we have $||u||=1$, and hence $||xu-P_nxu||\le||x-P_nxu||$. Therefore, by Proposition \ref{Calc}-(\ref{Calc2}), we have
$$\lambda(F)=\lim_{n\to+\infty}\sup_{x\in E,u-\mbox{\scriptsize unitary}}||xu-P_nxu||\le\lim_{n\to+\infty}\sup_{x\in E}||x-P_nx||=\lambda(E).$$
The opposite inequality $\lambda(E)\le\lambda(F)$ follows from $E\subseteq F$.
\end{proof}

\begin{proposition} Let $E\subseteq l^2(\mathcal A)$ be a bounded set, let $u\in\mathcal A$ be a unitary and let $\mu$ stands for any of Kuratowski, Hausdorff or Istr\u{a}\c{t}escu MNC. Then
$$\mu^*(Eu)=\mu^*(E).$$
\end{proposition}

\begin{proof} First, observe that given a normal state $\f$ on $\mathcal A$ and a unitary $u\in\mathcal A$, the mapping $\f^u$, $\f^u(x)=\f(u^*xu)$ is also a normal state. Obviously, $\f^u(1)=1$ and $\f^u(x)\ge0$ whenever $x\ge0$. Hence, it suffices to prove that $\f^u$ is normal. Let $x_\a$ be an increasing net with the least upper bound $x$. Then $ux_\a u^*$ is also an increasing net, bounded by $uxu^*$. Thus, its least upper bound is less then $uxu^*$. Moreover, it is equal to $uxu^*$ by interchanging roles. Therefore
$$\sup\f^u(x_\a)=\sup\f(ux_\a u^*)=\f(uxu^*)=\f^u(x).$$

Let $\e>0$ be arbitrary. By (\ref{MNCdef}) there is a semi-norm $p_{\f,y}\in P$ such that $[\mu(Eu)](p_{\f,y})>\mu^*(Eu)-\e$. Next, we have $p_{\f,y}(xu)=p_{\f^u,yu}$, where $yu^*=(\eta_1u^*,\eta_2u^*,\dots)$. Indeed
$$p_{\f,y}(xu)^2=\sum_{j=1}^{+\infty}|\f(\eta_j^*\xi_ju)|^2=\sum_{j=1}^{+\infty}|\f(u^*(\eta_ju^*)^*\xi_ju)|^2=
p_{\f^u,yu^*}(x)^2.$$
(Note that the pair $(\f^u,yu^*)$ trivially satisfies (\ref{uslov za polunormu}).) Therefore, $[\mu(E)](p_{\f^u,yu^*})=[\mu(Eu)](p_{\f,y})>\mu^*(Eu)-\e$ and hence $\mu^*(E)\ge\mu^*(Eu)$. The opposite inequality follows by $E=(Eu)u^{-1}$.
\end{proof}

\begin{remark}\label{RemarkBalanced}
We don't know whether Kuratowski, Hausdorff and Istr\u{a}\c{t}escu MNCs are stable with respect to balanced hull.
\end{remark}

Now, we are going to derive an upper bound for the MNC $\lambda$. Namely, in a special case, where $\mathcal A=B(H)$, we can obtain the exact position of $\lambda$ in the inequality chain (\ref{chainMNC}) for balanced sets.

\begin{theorem}\label{lambda_leq_I}
Let $\mathcal A=B(H)$. Let $E\subseteq l^2(\mathcal A)$ be an $\mathcal A$-balanced set. Then
$$\lambda(E)\leq\sqrt{||E||I^*(E)},$$
where $||E||=\sup_{x\in E}||x||$.
\end{theorem}

\begin{proof}Let $P_k$ denote the projection to the first $k$ coordinates, i.e.\
$P_k(\xi_1,\xi_2,\dots)\allowbreak=(\xi_1,\dots,\xi_k,0,0,\dots)$. It is well known that all $P_k$ are
"compact". If $\inf_{k\ge1}||(I-P_k)E||=0$, then from Proposition
\ref{chi_leq_lambda} and (\ref{chainMNC}), it follows $\lambda^*(E)=\chi^*(E)=I^*(E)=0$. So, let
$$\delta=\inf_{k\ge1}||(I-P_k)E||>0.$$
Then immediately, $\delta\le||E||$. Choose $\varepsilon>0$ such that $\varepsilon<\delta^2$. Define the
sequence of projections $Q_n\in\{P_1,P_2,\dots\}$ and the sequences of vectors $x_n$ and $z_n\in l^2(\mathcal
A)$ in the following way. Let $Q_0=0$. If $Q_{n-1}$ is already defined, there is $x_n\in E$ such that
$||x_n||\geq||(I-Q_{n-1})x_n||>C_1$, where
$C_1=\frac{1}{2}\left(\delta+\sqrt{\delta^2-\varepsilon/2}\right)$. Since
$\lim_{k\to+\infty}||(I-P_k)(I-Q_{n-1})x_n||=0$, there is a positive integer $k_n$ such that
$||(I-P_{k_n})(I-Q_{n-1})x_n||<C_2$, where
$C_2=\min\left\{\frac{\varepsilon}{2||E||},\frac{1}{2}\left(\delta-\sqrt{\delta^2-\varepsilon/2}\right)\right\}$. Define
$Q_n=P_{k_n}$ and
\begin{equation}\label{define_z}
z_n=Q_n(I-Q_{n-1})x_n.
\end{equation}

The sequences $x_n$ and $z_n$ have the following properties:

Firstly, by definition, there hold the inequalities
\begin{equation}\label{property1}
||(I-Q_n)(I-Q_{n-1})x_n||<C_2,
\end{equation}
\begin{equation}\label{property1a}
||z_n||\le||x_n||\le||E||,
\end{equation}
\begin{equation}\label{property1b}
||z_n||\ge||(I-Q_{n-1})x_n||-||(I-Q_n)(I-Q_{n-1})x_n||>C_1-C_2.
\end{equation}

Secondly,
\begin{equation}\label{property2}
\left<z_n,x_n\right>=\left<z_n,z_n\right>.
\end{equation}
Indeed, since $z_n=Q_n(I-Q_{n-1})x_n$, we have
\begin{align*}
\left<z_n,x_n\right>&=\left<Q_n(I-Q_{n-1})x_n,x_n\right>=\\
                    &=\left<Q_n(I-Q_{n-1})x_n,(I-Q_{n-1})Q_nx_n\right>=\left<z_n,z_n\right>.
\end{align*}

Thirdly, for $m>n$ we have
\begin{equation}\label{property3}
||\left<z_m,x_n\right>||<C_2||E||.
\end{equation}
Indeed, for such $m$ and $n$ we have $Q_{n-1}\le Q_n\le Q_{m-1}$, i.e.\ $I-Q_{m-1}\le I-Q_n\le I-Q_{n-1}$,
implying $I-Q_{m-1}=(I-Q_{m-1})(I-Q_n)(I-Q_{n-1})$, and thus
\begin{align*}
\left<z_m,x_n\right>&=\left<(I-Q_{m-1})z_m,x_n\right>=\\
    &=\left<z_m,(I-Q_{m-1})(I-Q_n)(I-Q_{n-1})x_n\right>=\\
    &=\left<z_m,(I-Q_n)(I-Q_{n-1})x_n\right>.
\end{align*}
Therefore, by (\ref{property1}) and (\ref{property1b})
$$||\left<z_m,x_n\right>||\le||z_m||\cdot||(I-Q_n)(I-Q_{n-1})x_n||\le C_2||E||.$$

Let us construct a semi-norm $p$, continuous in $\tau$, and a totally discrete sequence from $E$. Since by
(\ref{property1b}) $||z_n||^2=||\left<z_n,z_n\right>||>(C_1-C_2)^2$, we can choose a normal state $\f$ and
$\upsilon_j$, $\nu_j\in\mathcal A$ according to \cite[Lemma 4.6.]{KeckicLazovic}), such that
\begin{equation}\label{phi_choosed}
\f(\upsilon_n^*\left<z_n,z_n\right>\nu_n)>(C_1-C_2)^2.
\end{equation}
Consider the semi-norm $p$ given by
$$p(x)=\sqrt{\sum_{j=1}^{+\infty}|\f(\left<z_j\upsilon_j,x\right>)|^2}.$$
By (\ref{define_z}) there is a sequence $\zeta_j\in\mathcal A$ such that
$$z_n=(0,\dots,0,\zeta_{k_{n-1}+1},\dots,\zeta_{k_n},0,\dots).$$
Define $\omega_j=\zeta_j\upsilon_n/\f(\upsilon_n^*\zeta^*_j\zeta_j\upsilon_n)^{1/2}$, for $k_{n-1}+1\leq
j\leq k_n$. Obviously $\f(\omega_j^*\omega_j)=1$. Also, for $x=(\xi_1,\xi_2,\dots)$ we have
\begin{align*}
\left|\f(\left<z_n\upsilon_n,x\right>)\right|^2&=\Big|\sum_{j=k_{n-1}+1}^{k_n}\f(\upsilon_n^*\zeta^*_j\zeta_j\upsilon_n)^{1/2}\f(\omega_j^*\xi_j)\Big|^2\le\\
    &\le\sum_{j=k_{n-1}+1}^{k_n}\f(\upsilon_n^*\zeta^*_j\zeta_j\upsilon_n)\sum_{j=k_{n-1}+1}^{k_n}\left|\f(\omega_j^*\xi_j)\right|^2=\\
    &=\f(\upsilon^*_n\left<z_n,z_n\right>\upsilon_n)\sum_{j=k_{n-1}+1}^{k_n}\left|\f(\omega_j^*\xi_j)\right|^2.
\end{align*}
Including (\ref{property1a}) we obtain
$\f(\upsilon^*_n\left<z_n,z_n\right>\upsilon_n)\le||\upsilon^*_n\left<z_n,z_n\right>\upsilon_n||=||z_n||^2\le||E||^2$
and hence
$$p(x)^2=\sum_{n=1}^{+\infty}|\varphi(\left<z_n\upsilon_n,x\right>)|^2\le||E||^2\sum_{j=1}^{+\infty}|\f(\omega_j^*\xi_j)|^2=||E||^2
p_{\f,\omega_1,\dots,\omega_n,\dots}(x)^2.$$
Thus, we conclude that $p$ is well defined and also that it is continuous with respect to $\tau$.

Also, $E$ is $\mathcal A$-balanced, so $x_n\nu_n\in E$. Finally we shall prove that $x_n\nu_n$ is a totally
discrete sequence. Indeed, for $m>n$ we have
\begin{align*}
p(x_m\nu_m-x_n\nu_n)&\ge\left|\f(\left<z_m\upsilon_m,x_m\nu_m-x_n\nu_n\right>)\right|\ge\\
    &\ge\left|\f(\upsilon_m^*\left<z_m,x_m\right>\nu_m)\right|-\left|\f(\upsilon_m^*\left<z_m,x_n\right>\nu_n)\right|.
\end{align*}
However, by (\ref{property2}) and (\ref{phi_choosed}),
$$\left|\f(\upsilon_m^*\left<z_m,z_m\right>\nu_m)\right|>(C_1-C_2)^2$$
and, by (\ref{property3})
$$|\f(\upsilon_m^*\left<z_m,x_n\right>\nu_n)|\le||\left<z_m,x_n\right>||<C_2||E||.$$
Therefore
$$p(x_m\nu_m-x_n\nu_n)>(C_1-C_2)^2-C_2||E||\geq\delta^2-\varepsilon$$
and
$$p_{\f,\omega_1,\dots,\omega_n,\dots}(x_m\nu_m-x_n\nu_n)>\frac{\delta^2-\varepsilon}{||E||}.$$
For $\varepsilon\in(0,\delta^2)$, we have $I^*(E)\geq\frac{\delta^2-\varepsilon}{||E||}$ and hence
$I^*(E)\geq\frac{\delta^2}{||E||}=\frac{\lambda(E)^2}{||E||}$. Thus, $\lambda(E)\leq\sqrt{||E||I^*(E)}$.
\end{proof}

\begin{remark} The preceding proof is adapted proof of our earlier result \cite[Theorem 4.10]{KeckicLazovic}.
\end{remark}

\begin{corollary} On the standard Hilbert module $l^2(B(H))$ over the algebra $B(H)$ of all bounded operators on a Hilbert space $H$, there holds
$$\chi^*(E)\leq\lambda(E)\leq\sqrt{||E||I^*(E)}\leq\sqrt{||E||\alpha^*(E)}\leq\sqrt{2||E||\chi^*(E)},$$
for any balanced set, i.e.\ a set for which $x\in E$, $u$-unitary implies $xu\in E$.
\end{corollary}

Finally, we discuss some relationship between the MNC of an arbitrary adjointable operator $\lambda_0$
introduced in Definition \ref{DefinicijaLambda0}, and corresponding MNCs derived from $\alpha^*$, $\chi^*$ and $I^*$.

\begin{definition}
Let $\mathcal A$ is an arbitrary $W^*$-algebra and let $T\in B^a(l^2(\mathcal A))$ be an adjointable
operator. The functions $\alpha_0^*$, $\chi_0^*$, $I_0^*:B^a(l^2(\mathcal A))\rightarrow[0,+\infty)$ defined
by
$$\alpha_0^*(T)=\alpha^*(T(B_1)),\quad\chi_0^*(T)=\chi^*(T(B_1)),\quad I_0^*(T)=I^*(T(B_1))$$
are called, respectively, Kuratowski, Hausdorff and Istr\u{a}\c{t}escu measure of noncompactness of the
operator $T$.
\end{definition}

\begin{proposition}\label{osobine_T}
Let $\mathcal A$ be an arbitrary $W^*$-algebra, let $T$, $S\in B^a(l^2(\mathcal A))$ and let $\mu$ stands for any of MNCs $\alpha$, $\chi$, $I$. Then

\begin{itemize}

\item[(a)] All $\alpha^*_0$, $\chi^*_0$ and $I^*_0$ are subadditive and positivel homogeneous, i.e.\ there holds
$$\mu^*_0(T+S)\le\mu^*_0(T)+\mu^*_0(S),\quad\mu^*_0(cT)=c\mu^*(T),\quad\mbox{for all }c>0.$$

\item[(b)] The functions $\alpha_0^*$, $\chi_0^*$ and $I_0^*$ are equivalent to each other,
that is,
$$\chi_0^*(T)\leq I_0^*(T)\leq\alpha_0^*(T)\leq2\chi_0^*(T).$$
Also, there holds $\chi_0^*(T)\le\lambda_0(T)$.

\item[(c)] $\chi^*_0(T)$, $\lambda_0(T)\le||T||$ and $\alpha^*_0(T)$, $I^*_0(T)\le2||T||$.

\item[(d)] If $T$ is "compact", i.e.\ $T$ belongs to the closed linear space generated by $x\mapsto z\skp yx$,
then $\lambda_0(T)=\chi_0(T)=\alpha_0(T)=I(T)=0$. In general, the converse might not hold.

\item[(e)] $\mu_0^*(T+K)=\mu_0^*(T)$, as well as
$\lambda_0(T+K)=\lambda_0(T)$ for all ''compact'' operators $K$.
\end{itemize}
\end{proposition}

\begin{proof} Part (a) follows easily from Theorem \ref{osobine mere nekompaktnosti}-(d), whereas part (b) follows from
(\ref{chainMNC}) and Proposition \ref{chi_leq_lambda}.

Since $T(B_1)\subseteq B(0;||T||)=||T||B_1$, it follows $\lambda_0(T)\le||T||$ according to Proposition
\ref{UnitBall}. Other inequalities in part (c) follows from part (a).

If $T$ is "compact", then $T$ is norm limit of finite rank operators. Hence $T(B_1)$ is $\mathcal
A$-precompact. Therefore $\lambda_0(T)=0$. This, together with (b) proves (d). The converse does not always
hold due to \cite[Example 5.1.]{KeckicLazovic}.

Finally, (e) follows from (d).
\end{proof}

In the case $\mathcal A=B(H)$, we can obtain more.

\begin{proposition} Let $\mathcal A=B(H)$ and let $T\in B^a(l^2(B(H)))$. Then

\begin{itemize}

\item[(a)] There holds
$$\chi_0^*(T)\leq\lambda_0(T)\leq\sqrt{||T||I_0^*(T)}\leq\sqrt{||T||\alpha_0^*(T)}\leq\sqrt{2||T||\chi_0^*(T)}.$$

\item[(b)] $\mu_0^*(T)=0$, $\mu\in\{\alpha,\chi,I\}$ iff $\lambda_0(T)=0$ iff $T$ is a ''compact'' operator;
\end{itemize}
\end{proposition}

\begin{proof}
Part (a) follows from Proposition \ref{osobine_T}-(b) and Theorem \ref{lambda_leq_I}, since $T(B_1)$ is a
balanced set ($y=Tx\in T(B_1)$ implies $yu=T(xu)\in T(B_1)$).

Part (b) follows from part (a) and Proposition \ref{osobine_T}-(d). Indeed, if any of four MNCs annihilate
$T$, then, by part (a), $\alpha^*_0(T)=0$, and hence, $T(B_1)$ is totally bounded in the topology $\tau$. By
\cite[Theorem 4.10.]{KeckicLazovic}, $T$ is "compact".
\end{proof}

\section{Three questions}

\begin{question} It is easy to obtain the following: Let $E\subseteq l^2(\mathcal A)$ be bounded. For any $\e>0$ there is an
$\mathcal A$-precompact set $C_\e$ such that $E\subseteq C_\e+(\lambda(E)+\e)B$. (For instance $C_\e=E\cap M$
for a suitable free finitely generated $M$.)

Is it possible to get something stronger: There is an $\mathcal A$-precompact set $C$ such that $E\subseteq
C+\lambda(E)\cdot B$?
\end{question}

\begin{question} Among all properties of MNCs on a Banach space, it turns out that the most important is
$\mu(\co E)=\mu(E)$. This was proved in this note for $\lambda$ if we $\co E$ regard as a real field convex
hull, i.e.\ $\co E=\{\sum_1^nc_jx_j\;|\;\sum c_j=1, c_j\in{\mathbf R},x_j\in E\}$. However, there is a notion
of $\mathcal A$-convex hull (see for instance \cite{Magajna})
$$\co_{\mathcal A}E=\Big\{\sum_{j=1}^na_j^*x_ja_j\;|\;x_j\in E, a_j\in\mathcal A,\sum a_j^*a_j=1\Big\}.$$

Is it possible to obtain $\lambda(\co_{\mathcal A}E)=\lambda(E)$?
\end{question}

\begin{question} As it was mentioned in Remark \ref{RemarkBalanced}, we ask for the following: Is it true $\mu^*(E)=\mu^*(F)$, where $F$ denotes the balanced hull of $E$, i.e.\ $F=\bigcup_uEu$ (the union runs through all unitaries $u$), and $\mu$ denotes any of Kuratowski, Hausdorff or Istr\u{a}\c{t}escu MNC? 
\end{question}

\smallskip

%{\it Problem 2.} It seems that Lipschitz continuity itself is not enough for the given mapping $T$ to maps
%$A$-precompact sets into $A$-precompact (and furthermore to ensure
%$\lambda(T(E))\le\mathop{\mathrm{Lip}}(T)\lambda(E)$).

%What might be the additional assumption? Maybe $T$ is $A$-differentiable, i.e.\ for any $x$ there is $L\in
%B^a(l^2(\mathcal A))$ such that
%
%$$T(x+h)=Tx+Lh+o(||h||)?$$

\subsection*{Acknowledgement}
The authors was supported in part by the Ministry of education and science, Republic of Serbia, Grant
\#174034.

\bibliographystyle{plain}
\bibliography{MeasuresBibl}

\end{document}